\documentclass[12pt]{article}

\usepackage[utf8]{inputenc}
\usepackage{amsfonts,amssymb}
\usepackage{amsmath,amsthm}
\usepackage[dvipsnames]{xcolor}

\usepackage[margin=1in]{geometry}

\newtheorem{theorem}{Theorem}[section]
\newtheorem{lemma}[theorem]{Lemma}
\newtheorem{proposition}[theorem]{Proposition}
\newtheorem{corollary}[theorem]{Corollary}
\newtheorem{definition}[theorem]{Definition}
\newtheorem{question}[theorem]{Question}

\newcommand\F{\mathbb{F}}
\newcommand\R{\mathbb{R}}

\newcommand\cF{\mathcal{F}}
\newcommand\cJ{\mathcal{J}}
\newcommand\cM{\mathcal{M}}
\newcommand\cS{\mathcal{S}}

\newcommand\al{\alpha}
\newcommand\bt{\beta}
\newcommand\Dl{\Delta}
\newcommand\lm{\lambda}
\newcommand\Lm{\Lambda}

\newcommand\ad{\mathrm{ad}}

\title{Radicals in primitive axial algebras}
\author{
A.~Mamontov\footnote{This author was suppoorted by the Mathematical Center in Akademgorodok under the agreement No. 075-15-2025-348 with the Ministry of Science and Higher Education of the Russian Federation.}, 
S.~Shpectorov and 
V.~Zhelyabin\footnote{The research was carried out within
the framework of the Sobolev Institute of Mathematics
state contract (project FWNF-2026-0017).}}

\begin{document}
\maketitle

\newcommand{\Addresses}{{
		\bigskip\noindent
		\footnotesize
        Andrey~Mamontov, \textsc{Sobolev Institute of Mathematics, Novosibirsk, Russia;}\\\nopagebreak
		\textit{E-mail address: } \texttt{mamontov@math.nsc.ru}
		\\\nopagebreak	
        ORCID:\textsc{0000-0002-0324-5287}
        
        \medskip\noindent
        Sergey~Shpectorov, \textsc{School of Mathematics, University of Birmingham, Birmingham, UK;}\\\nopagebreak
        \textit{E-mail address: } 
        \texttt{s.shpectorov@bham.ac.uk}
        \\\nopagebreak
        ORCID:\textsc{0000-0001-6202-5885}
		
		\medskip\noindent
        Victor~Zhelyabin, \textsc{Sobolev Institute of Mathematics, Novosibirsk, Russia;}\\\nopagebreak
		\textit{E-mail address: } \texttt{vicnic@math.nsc.ru}
		\\\nopagebreak
		ORCID:\textsc{0000-0001-9371-6363}
        	
		\medskip
}}

\section{Introduction}

The class of axial algebras was introduced in \cite{hrs,hrs1} as a broad 
generalization of the class of Majorana algebras of Ivanov \cite{i}. These are 
commutative non-associative algebras generated by (primitive) idempotents called axes, 
whose adjoint action is semisimple and satisfies a prescribed fusion law. (See the 
exact definitions in Section \ref{axial algebras}.) Axial algebras is currently a 
very active area of research, with many strong results obtained particularly about 
two specific classes of primitive axial algebras, the algebras of Monster type 
$(\al,\bt)$, generalizing Majorana algebras, and the more narrow class of algebras of 
Jordan type $\eta$ including the axial Jordan algebras and the Matsuo algebras 
corresponding to $3$-transposition groups. For a recent survey of axial algebras, see 
\cite{ms}. 

The structure theory of primitive axial algebras was developed in \cite{kms}. In 
particular, this paper introduced the concept of the radical of an axial algebra $A$, 
as the largest ideal $R(A)$ not containing any of the generating axes. When $A$ 
is endowed with a Frobenius form $(\cdot,\cdot)$, a bilinear form associating with the algebra 
product, the axial radical $R(A)$ is contained in the Frobenius form radical 
$A^\perp$, with the equality taking place if and only if all generating axes are 
non-singular for the form. 

In this note we continue developing the structure theory of primitive axial algebras 
by comparing the above two radicals, $R(A)$ and $A^\perp$, with the Jacobson radical 
$J(A)$, which we define simply as the intersection of all maximal ideals of $A$. 
Our main result is as follows.

\begin{theorem} \label{main}
If $A$ is a primitive axial algebra with a Frobenius form $(\cdot,\cdot)$ then 
$R(A)\subseteq J(A)\subseteq A^\perp$.
\end{theorem}

Clearly, this means that if all generating axes are non-singular with respect to the 
Frobenius form, which happens in the vast majority of known examples, then 
$R(A)=J(A)=A^\perp$.

Note that in this theorem we allow the Frobenius form to be zero, which means our statement covers all primitive axial algebras. When the Frobenius form is zero, $A^\perp=A$, but we still have the non-obvious statement $R(A)\subseteq J(A)$.

\medskip
In the remainder of the text, we discuss the possibility of the three radicals not 
being equal. While the equality $J(R)=A^\perp$ heavily depends on the choice of the 
Frobenius form, it is possible that the equality $R(A)=J(A)$ is more universal and we 
suggest various ideas that may help either to prove or disprove this assertion. We 
also show that the hull-kernel topology on the set of maximal ideals of a primitive 
axial algebra $A$ is discrete. Towards the end, we list some open questions.

\section{Axial algebras} \label{axial algebras}

In this section we provide the necessary background information on axial algebras. 
First of all, all algebras in this note are non-associative, that is, associativity 
is not assumed. 

\begin{definition}
A \emph{fusion law} is a pair $(\cF,\ast)$, where $\cF$ is a set (typically a subset 
of a field $\F$) and $\ast$ is a binary operation on $\cF$ with values in the set 
$2^{\cF}$ of all subsets of $\cF$, that is, a map $\cF\times\cF\to 2^{\cF}$. 
\end{definition}

We will usually simply write $\cF$ for the fusion law $(\cF,\ast)$, i.e., the operation $\ast$ will be tacitly assumed.

If $A$ is a commutative algebra over a field $\F$ and $a\in A$, we write $\ad_a$ 
for the \emph{adjoint action} of $a$ on $A$, that is, the linear map $A\to A$ given 
by $u\mapsto au$. Since $A$ is commutative, we do not need to distinguish between the 
left and right adjoints, and so we simply speak of the adjoin action.

For $\lm\in\F$, we write $A_\lm(a)$ for the $\lm$-eigenspace of $\ad_a$, that is, 
$$
A_\lm(a)=\{u\in A\mid au=\lm u\}.
$$
Note that $A_\lm(a)$ is non-zero if and only if $\lm$ is an eigenvalue of $\ad_a$. 
For a subset $\Lm\subseteq\F$, we set 
$$
A_\Lm(a)=\oplus_{\lm\in\Lm}A_\lm(a).
$$

\begin{definition}
Suppose that $\cF\subseteq\F$ is a fusion law. A non-zero idempotent $a\in A$ is an 
\emph{axis} for the fusion law $\cF$ if 
\begin{enumerate}
\item[(a)] $A=A_\cF(a)$, and
\item[(b)] $A_\lm(a)A_\mu(a)\subseteq A_{\lm\ast\mu}(a)$.
\end{enumerate}
\end{definition}

In other words, the adjoint action of $a$ must be semisimple with all eigenvalues in 
$\cF$. Furthermore, the operation $\ast$ on $\cF$ controls products of eigenvectors 
of $\ad_a$.

Note that the above definition requires $\cF$ to contain $1$. Indeed, 
$\ad_a(a)=a^2=a$, which means that $a\in A_1(a)$, and so $A_1(a)$ is non-zero. In 
what follows we always assume that $1\in\cF$.

\begin{definition}
An axis $a\in A$ is said to be \emph{primitive} if $A_1(a)$ is $1$-dimensional, 
spanned by $a$.
\end{definition}

We are now prepared for the main definition.

\begin{definition}
Suppose that $\cF\subseteq\F$ is a fusion law. An axial algebra for the fusion law 
$\cF$ is a pair $(A,X)$, where $A$ is a commutative algebra over $\F$ and $X\subset 
A$ is a set of axes for $\cF$ generating $A$.

An axial algebra $(A,X)$ is said to be \emph{primitive} if all axes in $X$ are 
primitive.
\end{definition}

While non-primitive axial algebras may also be interesting, large parts of the theory 
of axial algebras depend on the primitivity condition. Hence below we will always 
assume by default that all axial algebras we consider are primitive. 

We will often speak of an axial algebra $A$ without mentioning the specific set of 
generating axes $X$. However, $X$ is always assumed, as $A$ may in general have many 
different generating sets of axes for different fusion laws. 

The structure theory for primitive axial algebras was developed in \cite{kms}. In 
particular, this paper introduced the concept of the radical of a primitive axial 
algebra.

\begin{definition}
The \emph{radical} $R(A)$ of a primitive axial algebra $A$ is the unique largest 
ideal of $A$ not containing any axes from $X$.
\end{definition}

The existence of the radical follows from Lemma 4.2 in \cite{kms}, see the discussion 
after Definition 4.3 there. As we consider several different radicals in this note, 
we will refer to $R(A)$ as the \emph{axial radical} of $A$.

Every ideal of $A$ is invariant under the adjoint action of an axis $a$. This leads 
to the following result that we will need later. First, we introduce the following 
notation and related terminology. Since $A=A_\cF(a)=\oplus_{\lm\in\cF}A_\lm(a)$, 
every $u\in A$ can be uniquely decomposed as 
$$
u=\sum_{\lm\in\cF}u_\lm,
$$
where $u_\lm\in A_\lm(a)$ for all $\lm\in\cF$. We will call $u_\lm$ the 
\emph{components} of $u$ with respect to the axis $a$.

\begin{lemma} \label{component}
If $I$ is an ideal of $A$ and $u\in I$ then each component $u_\lm$ of $u$ with 
respect to an axis $a\in A$ is also contained in $I$.
\end{lemma}

We will be assuming that $A$ is endowed with a bilinear forms associating with the 
algebra product.

\begin{definition}
A \emph{Frobenius form} on an axial algebra $A$ is a bilinear form $(\cdot,\cdot)$ 
such that 
$$
(uv,w)=(u,vw)
$$
for all $u,v,w\in A$.
\end{definition}

We note that we are really only interested in the non-zero Frobenius forms, and in 
most sources this is explicitly assumed. In the current note, for uniformity, it is 
convenient to assume that our arbitrary axial algebra admits a Frobenius form, which 
can only be guarantied if we allow the form to be zero. 

Let us mention some basic facts related to this concept. First of all, it is well 
known that a Frobenius form is always symmetric, that is,
$$
(u,v)=(v,u)
$$
for all $u,v\in A$.

\begin{lemma} \label{orthogonal}
For an axis $a\in A$ and $\lm,\mu\in\cF$, $\lm\neq\mu$, the eigenspaces $A_\lm(a)$ 
and $A_\mu(a)$ are orthogonal with respect to the Frobenius form on $A$.
\end{lemma}

We use the notation $A^\perp$ for the radical of the Frobenius form. That is, 
$$
A^\perp=\{u\in A\mid(u,v)=0\mbox{ for all }v\in A\}.
$$

We note that $A^\perp$ is an ideal of $A$, and it turns out that, under some mild 
extra assumptions, $A^\perp$ coincides with the axial radical $R(A)$. Indeed, Lemma 
\ref{orthogonal} implies that $a\in A_1(a)$ is orthogonal to all of 
$A_{\cF-\{1\}}(a)$. Hence, we have the following.

\begin{lemma} \label{axis}
If an axis $a\in A$ is primitive then $a\in A^\perp$ if and only if $(a,a)=0$.
\end{lemma}

This, in turn, implies the desired claim.

\begin{lemma} \label{perp}
If the Frobenius form on a primitive axial algebra $A$ has the property that 
$(a,a)\neq 0$ for all $a\in X$ then the radical $R(A)$ of the algebra coincides with 
the radical of the form, $A^\perp$.
\end{lemma}

Finally, let us discuss the direct sum construction for axial algebras. Recall that, 
given algebras $A_j$, $j\in J$, we can define the direct sum algebra 
$$A=\oplus_{j\in J}A_j,$$
where $A_iA_j=0$ for all $i,j\in J$, $i\neq j$. 

First let us observe where primitive axes can lie in a direct product algebra.

\begin{lemma} \label{primitive}
If $a$ is a primitive axis in $A=\oplus_{j\in J}A_j$ then $a\in A_j$ for some 
$j\in J$.
\end{lemma}

Indeed, decompose $a=\sum_{j\in J} a_j$, with each $a_j\in A_j$, and note that each 
$a_j$ is contained in the $1$-eigenspace of $\ad_a$. Clearly, this means that only 
one $a_j$ is non-zero and so $a=a_j\in A_j$.

Suppose now that each $A_j$ is an axial algebra for the same fusion law 
$\cF$, that is, we have a generating set of axes $X_j$ in each 
$A_j$. Set 
$$X=\cup_{j\in J}X_j.$$
Comparing with Lemma \ref{primitive}, this is the only possible way to introduce on 
$A$ the structure of a primitive axial algebra. Hence, the natural question to ask 
is: when is the pair $(A,X)$ an axial algebra for the fusion law $\cF$?

\begin{lemma} \label{direct sum}
When $|J|\geq 2$, the pair $(A,X)$ is an axial algebra for the 
fusion law $\cF$ if and only if $0\in\cF$ and, additionally, $0\in 
0\ast 0$.
\end{lemma}

Coincidentally, the same condition, $0\in\cF$ and $0\in 0\ast 0$, is required for 
$A\not\cong\F$ to be unital, i.e., to have a multiplicative identity element.

\section{Blocks} \label{blocks}

In \cite{kms}, sum decompositions of primitive axial algebras, slightly generalizing 
direct sum decompositions, were investigated. In this note we take a more general 
viewpoint. 

\begin{definition} \label{block}
For an axial algebra $(A,X)$ and an axis $a\in X$, the \emph{block} $I_a$ of $A$ 
corresponding to $a$ is defined as the ideal generated by $a$, i.e., the smallest 
ideal of $A$ containing $a$.
\end{definition}

As the sum of ideals is again an ideal, we have the following observation.

\begin{lemma} \label{block decomposition}
We have that $A=\sum_{a\in X}I_a$.
\end{lemma}

Clearly, the right side is an ideal, and hence a subalgebra containing all of $X$, so 
it must coincide with $A$, as $X$ generates $A$.

\begin{definition}
We call an axial algebra \emph{decomposable} if it is the sum of its proper ideals.
\end{definition}

An indecomposable algebra has a unique maximal ideal, the sum of all its proper 
ideals.

In view of Lemma \ref{block decomposition}, if an axial algebra $A$ is indecomposable 
then $A=I_a$ for some $a\in X$. We will now show that in the primitive case the 
converse is also true. First, we broadly generalize Lemma \ref{primitive}.

\begin{lemma} \label{sum}
Suppose that an algebra $A$ is decomposed as a sum of its ideals:
$$
A=\sum_{j\in J}I_j.
$$
If $a\in A$ is a primitive axis then $a\in I_j$ for some $j\in J$.
\end{lemma}

\begin{proof}
{\it Ad absurdum}, suppose that $a\notin I_j$ for all $j\in J$. Take some $j\in J$ 
and $u\in I_j$. Note that by the primitivity of $a$, the component $u_1$, taken with 
respect to $a$, is a multiple of $a$. In view of Lemma \ref{component}, $u_1=0$ since 
$a\notin I_j$. It follows that $u\in A_{\cF-\{1\}}(a)$, and so $I_j\subseteq A_{\cF-
\{1\}}(a)$. Since this is true for all $j\in J$, we conclude that all $I_j$ are 
contained in $A_{\cF-\{1\}}(a)$, and so $A=\sum_{j\in J}I_j\subseteq 
A_{\cF-\{1\}}(a)$; clearly a contradiction. 
\end{proof}

Consequently, we have the following.

\begin{corollary}
If $A=I_a$ for some $a\in X$ then $A$ is indecomposable.
\end{corollary}

As a word of warning, we cannot claim that each block $I_a$ is itself an axial 
algebra. Indeed, the axes from $X\cap I_a$ may in principle generate a proper 
subalgebra of $I_a$. We also cannot claim that every block is indecomposable, as we cannot assume that every ideal of a block is an ideal of the whole algebra $A$.

Can distinct blocks be contained in one another? This kind of questions can be 
formulated and investigated using the natural domination relation on the set of axes 
$X$. 

\begin{definition}
For axes $a,b\in X$, $a$ \emph{dominates} $b$ if $I_b\subseteq I_a$. Two axes $a$ and 
$b$ are \emph{equivalent} if they dominate each other, that is, if $I_a=I_b$.
\end{definition}

Hence, the above question amounts to asking: can dominance be non-symmetric; i.e., 
can it not be an equivalence relation?

The concept of the projection graph was introduced in \cite{kms}.

\begin{definition}
The \emph{projection digraph} of a primitive axial algebra $(A,X)$ has $X$ as its set 
of vertices. For $a,b\in X$, $a\neq b$, there is an arc from $a$ to $b$ if and only 
if the projection of $a$ to $A_1(b)$ is non-zero. That is, $a_1\neq 0$, where $a_1$ 
is the $1$-component of $a$ with respect to $b$. 
\end{definition}

As $a_1$ is a multiple of $b$, since $b$ is primitive, we have that $b\in I_a$, that 
is $a$ dominates $b$. As dominance is transitive, we have the following.

\begin{lemma}
The axis $a$ dominates all axes $b$ that can be reached from $a$ via oriented paths 
in the projection graph.
\end{lemma}

When two axes are connected by arcs in both directions, we will say that they are 
connected by an undirected \emph{edge}. Such two axes are clearly equivalent, and 
so we also have the following.

\begin{lemma}
If two vertices $a$ and $b$ are connected in the projection graph via a path 
consisting entirely of undirected edges then $a$ and $b$ are equivalent. 
\end{lemma}

Finally, let us mention what happens in the case where the Frobenius from has the 
property that all axes $a\in X$ are non-singular, i.e., $(a,a)\neq 0$ for all $a\in 
X$. In this case, an arc from $a$ to $b$ exists if and only if $(a,b)\neq 0$. That 
is, the projection digraph turns into a \emph{non-orthogonality} graph, as all arcs 
are parts of undirected edges. Hence all axes within the same connected component of 
the non-orthogonality graph are equivalent. Note that this still leaves the 
possibility of non-symmetric domination between axes from different components.

\section{Jacobson radical}

We start by repeating the definition already stated in the introduction.

\begin{definition}
The \emph{Jacobson radical} $J(A)$ of an axial algebra $A$ is the intersection of all 
maximal ideals of $A$. 
\end{definition}

If $A$ has no maximal ideals then we simply take $J(A)=A$.

\medskip
There seems to be no general consensus in the literature how to define the Jacobson 
radical for a general non-associative algebra. For the class of Jordan algebras, 
closely related to axial algebras,  McCrimmon \cite{McC} defined the quasi-regular 
Jacobson radical.  Zelmanov introduced the concept of a primitive unital Jordan 
algebra and showed that the Jacobson radical is the intersection of primitive 
ideals (cores of maximal quadratic ideals)  \cite{z}. Hogben and McCrimmon 
\cite{HMcC} extended this characterization of the Jacobson radical to arbitrary 
Jordan algebras. We wonder whether these definitions would simplify under the 
axial assumption and be closer to what we consider in this paper.

\medskip
In the remainder of this section, we prove half of Theorem \ref{main}.

\begin{proposition}
The axial radical $R(A)$ of a primitive axial algebra $A$ is contained in 
the Jacobson radical $J(A)$. 
\end{proposition}

\begin{proof}
By contradiction, suppose that the axial radical $R=R(A)$ of $A$ is not contained in 
$J(A)$. Then $A$ has a maximal ideal $I$ not containing $R$. Note that $R+I=A$, since 
$I$ is a maximal ideal. According to Lemma \ref{sum}, every $a\in X$ is contained in 
$R$ or in $I$. However, none of them can be in $R$ by the definition of the axial 
radical. Hence $X\subseteq I$ and so also $A\subseteq I$, as $X$ generates $A$; a 
contradiction.
\end{proof}

\section{Ideal complement} \label{complement}

We now start proving the other half of Theorem \ref{main}, concerning the Frobenius 
form radical $A^\perp$. In this section, we assume that the Frobenius from 
$(\cdot,\cdot)$ on a primitive axial algebra $A$ is non-degenerate, that is, 
$A^\perp=0$. By Lemma \ref{axis}, $(a,a)\neq 0$ for all $a\in X$, and this, in turn, 
implies, according to Lemma \ref{perp}, that the axial radical of $A$ coincides with 
$A^\perp$, and so it is also trivial. 

Suppose $U\neq 0$ is an ideal of $A$. Then $U$ must contain some axes from $X$, as 
$R(A)$ is trivial. Define $X_1=X\cap U$ and $X_2=X\setminus X_1$.

Recall from Lemma \ref{component} that, for $a\in X$ and $u\in U$, every component of 
$u$ with respect to $a$ is also contained in $U$. In particular, if $a\in X_2$ then 
the component $u_1$ is zero, as it must be a multiple of $a$. Now it follows from 
Lemma \ref{orthogonal} that $(a,u)=0$ for all $a\in X_2$ and $u\in U$.

Consider the magma $M_i$ of all products of the axes from $X_i$, $i=1,2$. 

\begin{lemma} \label{orthogonal to ideal}
For all $m\in M_2$, we have that $(U,m)=0$.
\end{lemma}

\begin{proof}
By contradiction, suppose that we have $m\in M_2$ that are not orthogonal to $U$. 
Select such a product $m$ with the smallest possible number of factors. We have seen 
that all $a\in X_2$ are orthogonal to $U$, so $m$ cannot be just an axis. Hence 
$m=m'm''$ for some $m',m''\in M_2$. Therefore, $(U,m'm'')=(Um',m'')=0$ since 
$Um'\subseteq U$ and $m''\in M_2$ has fewer factors compared to $m$. The 
contradiction proves the claim.
\end{proof}

\begin{lemma} \label{product zero}
For all $m_1\in M_1$ and $m_2\in M_2$, we have that $m_1m_2=0$.
\end{lemma}

\begin{proof}
Indeed, note that $M_1\subseteq U$. Therefore, $(A,m_1m_2)=(Am_1,m_2)=0$ by Lemma 
\ref{orthogonal to ideal} since $Am_1\subseteq U$. Since the Frobenius form is 
non-degenerate, we conclude that $m_1m_2=0$.
\end{proof}

Let $A_i$ be the subalgebra generated by $X_i$, $i=1,2$. Then $A_i$ is spanned by 
$M_i$. 

\begin{proposition} \label{decomposition}
We have that $A=A_1+A_2=A_1\oplus A_2$. Furthermore, $U=A_1$.
\end{proposition}

\begin{proof}
Clearly, for $u_1\in A_1$ and $u_2\in A_2$, we have that $u_1u_2=0$. This implies 
that $A_1+A_2$ is closed for products and so it is a subalgebra of $A$. Since 
$A_1+A_2$ contains $X=X_1\cup X_2$, which generates $A$, it follows that $A=A_1+A_2$. 
In particular, both $A_1$ and $A_2$ are ideals of $A$. This means that $I=A_1\cap 
A_2$ is also an ideal of $A$. Now note that $I^2=II\subseteq A_1A_2=0$, that is, $I$ 
is a nilpotent ideal. Consequently, $I$ contains no non-zero idempotents, and in 
particular, it does not contain any axes from $X$. Therefore, $I\subseteq R(A)=0$. We 
have shown that $A_1\cap A_2=0$, that is, $A=A_1\oplus A_2$.

Finally, if $U$ is bigger than $A_1$ then consider $J=U\cap A_2\neq 0$. We again have 
that $J$ is an ideal and, furthermore, $J$ contains no axes from $X_1$ since 
$J\subseteq A_2$. Also, $J$ contains no axes from $X_2$ since $J\subseteq U$ and 
$X_2$ is disjoint from $U$. Again, this means that $J$ is disjoint from the entire 
set of axes $X$, that is, $J$ is contained in $R(A)$, which is trivial. The 
contradiction proves that $U=A_1$, as claimed.
\end{proof}

\section{Semisimplicity} \label{semisimplicity}

We continue assuming that $A$ is a primitive axial algebra with a non-degenerate 
Frobenius from.

Recall that, for $a\in X$, the block $I_a$ is the ideal generated by $a$. 

\begin{lemma} \label{equal}
If $b\in X$ and $b\in I_a$ then $I_b=I_a$.
\end{lemma}

\begin{proof}
Clearly, $I_b\subseteq I_a$. By contradiction, suppose that $I_b\neq I_a$. Then 
$a\notin I_b$. We set $U=I_b$ and define $X_i$, $M_i$, and $A_i$, $i=1,2$, as in 
Section \ref{complement}. Then $a\in X_2\subset A_2$ and hence also $I_a\subset A_2$. 
However, this means (see Proposition \ref{decomposition}) that $b\in A_1\cap A_2=0$, 
which is a contradiction. 
\end{proof}

This means that the domination relation, as introduced in section \ref{blocks}, 
is symmetric in our algebra $A$, and hence dominance in $A$ is an equivalence 
relation. Let $C$ be the set of equivalence classes. For $c\in C$, we define $I_c$ 
as $I_a$ for an arbitrary axis $a\in c$.

\begin{proposition} \label{semisimple}
For $c,c'\in C$, $c\neq c'$, we have that $I_cI_{c'}=0$. Furthermore, each $(I_c,c)$ 
is a simple primitive axial algebra and 
$$
A=\oplus_{c\in C}I_c.
$$
\end{proposition}

\begin{proof}
The first claim follows from Proposition \ref{decomposition} since the axes from $c'$ 
are not contained in $U=I_c$. Furthermore, again by Proposition \ref{decomposition}, 
$I_c=A_1$ is a primitive axial algebra generated by $X_1=I_c\cap X=c$. Since $I_c$ is 
a direct summand of $A$, every ideal of $I_c$ is an ideal of $A$ and, clearly, if 
this ideal is proper in $I_c$ then it cannot contain any axes by Lemma \ref{equal}. 
Therefore, such an ideal would be in the radical $R(A)$ of $A$, hence trivial. We 
have shown that each $I_c$ is a simple primitive axial algebra.

By Lemma \ref{block decomposition}, the sum of all $I_c$, $c\in C$, is $A$. It 
remains to see that this sum is direct. Suppose that we have $0=\sum_{i=1}^k u_i$, 
where $u_i\in I_{c_i}$ for distinct $c_1,c_2,\ldots,c_k\in C$. Take $U=I_{c_k}$. 
Then, using the notation from Section \ref{complement}, $I_{c_k}=A_1$ and 
$I_{c_1},I_{c_2},\ldots,I_{c_{k-1}}\subseteq A_2$, as in Proposition \ref{decomposition}, 
and so $u_k=0$, since $A_1\cap A_2=0$. Clearly, this means that all $u_i$ are zero, 
and this yields the direct sum claim.
\end{proof}

We have shown that $A$ is a semisimple algebra, a direct sum of simple algebras.
We can now define $J_c$ as the sum of all blocks $I_{c'}$, $c'\neq c$. Clearly, 
$A=I_c\oplus J_c$, and so $J_c$ is a maximal ideal, since $A/J_c\cong I_c$ is simple. 
We note that the intersection of all the maximal ideals $J_c$ is trivial, which gives 
us the following.

\begin{corollary} \label{trivial Jacobson}
We have that $J(A)=0$.
\end{corollary}

\section{Corollaries}

Let us now list several applications of Proposition \ref{semisimple} and Corollary 
\ref{trivial Jacobson}. We start with the general statement completing the proof of 
Theorem \ref{main}.  

\begin{corollary} \label{modulo A perp}
Suppose that $A$ is an axial algebra admitting a Frobenius form. Then $A/A^\perp$ is 
semisimple and so $J(A)\subseteq A^\perp$. 
\end{corollary}

Indeed, $A/A^\perp$ is a primitive axial algebra with a non-degenerate Frobenius 
form, and hence the claim follows from Proposition \ref{decomposition} and Corollary 
\ref{trivial Jacobson}.

An interesting case arises when the fusion law $\cF$ does not allow direct sums. (See 
Lemma \ref{direct sum} and the discussion after it.)

\begin{corollary}
Suppose that either $0\notin\cF$ or $0\in\cF$ but $0\notin 0\ast 0$. If an axial 
algebra $A$ admits a Frobenius form then $A/A^\perp$ is a simple algebra.
\end{corollary}

Clearly, this follows from Corollary \ref{modulo A perp} and Lemma \ref{direct sum}.

We have already mentioned in the introduction that, whenever the Frobenius form 
satisfies the additional property that $(a,a)\neq 0$ for all $a\in X$ (i.e., all 
generating axes are non-singular for the Frobenius form), the chain of inclusions 
from Theorem \ref{main} turns into an equality: $R(A)=J(A)=A^\perp$. We now mention 
two specific classes of axial algebras where this is the case.

One of the most well studied classes of axial algebras is the class of algebras of 
Jordan type $\eta$, which are primitive axial algebras for the fusion law $\cJ(\eta)$ shown in Figure \ref{Jordan type}.
\begin{figure}[ht]
\renewcommand{\arraystretch}{1.6}
\begin{center}
\begin{tabular}{|c||c|c|c|}
\hline
$\star$ & $1$ &$0$&$\eta$\\
\hline  \hline
$1$&$1$& &$\eta$\\
\hline
$0$& &$0$&$\eta$\\
\hline
$\eta$&$\eta$&$\eta$&$1,0$\\
\hline
\end{tabular}
\end{center}
\caption{Jordan type fusion law $\cJ(\eta)$}\label{Jordan type}
\end{figure}
It was shown in \cite{hss2} that every algebra $A$ of Jordan type $\eta$ admits a 
unique Frobenius form, such that $(a,a)=1$ for every primitive axis $a\in A$. In 
particular, the axial radical $R(A)$ coincides with $A^\perp$. Combining this with 
Theorem \ref{main} yields the following.

\begin{corollary}
If $A$ is an algebra of Jordan type $\eta$ then $R(A)=J(A)=A^\perp$ and $A/A^\perp$ 
is semisimple.
\end{corollary}

The other much studied class of axial algebras is the class of algebras of Monster 
type $(\al,\bt)$. These are primitive axial algebras for the fusion law 
$\cM(\al,\bt)$ shown in Figure \ref{Monster type}.
\begin{figure}[ht]
\renewcommand{\arraystretch}{1.6}
\begin{center}
\begin{tabular}{|c||c|c|c|c|}
\hline
$\star$ & $1$ &$0$&$\alpha$& $\beta$\\
\hline \hline
$1$&$1$&&$\al$&$\bt$\\
\hline
$0$&&$0$&$\al$&$\bt$\\
\hline
$\al$&$\al$&$\al$&$1,0$&$\bt$\\
\hline
$\bt$&$\bt$&$\bt$&$\bt$&$1,0,\al$\\
\hline
\end{tabular}
\end{center}
\caption{Monster type fusion law $\cM(\al,\bt)$}\label{Monster type}
\end{figure}
There is no general result concerning the existence of a Frobenius form for algebras 
of Monster type $(\al,\bt)$. However, in an interesting subclass of this class, the 
existence of a Frobenius form is postulated as part of this subclass' axiomatics. 
Namely, Ivanov defined \emph{Majorana algebras}, modeled after the example of the 
Griess algebra for the Monster sporadic simple group, as the primitive algebras for 
the fusion law $\cM(\frac{1}{4},\frac{1}{32})\subseteq \F=\R$, admitting a 
positive-definite Frobenius form and satisfying a number of further axioms, which we 
omit here.

Clearly, since the Frobenius form on a Majorana algebra is positive-definite, it is 
non-degenerate. Hence we have the following.

\begin{corollary}
Every Majorana algebra is semisimple, i.e., a direct sum of simple Majorana algebras.
\end{corollary}

\section{Hull-kernel topology}

The hull-kernel topology is a natural topology on the set of all maximal ideals. It 
is defined as follows. Let $\cM$ be the set of all maximal ideals of $A$.

\begin{definition}
For a set $\cS\subseteq\cM$, its closure under the hull-kernel topology is 
$$
\bar\cS=\{J\in\cM\mid J\supseteq K\},
$$
where $K=\cap_{I\in\cS}I$.
\end{definition}

For general rings $A$, the hull-kernel topology may be quite non-trivial. We will see 
that this is not the case for primitive axial algebras $A$.

We start by reproving the claims in Proposition \ref{semisimple} under a weaker 
assumption. As in Section \ref{semisimplicity}, let $C$ be the set of equivalence 
classes for the mutual dominance equivalence relation on the set of generating axes 
$X$. Then $C$ can be used for indexing all blocks: for an equivalence class $c\in C$, 
we define $I_c=I_a$ for an arbitrary axis $a\in c$.

\begin{proposition}
Let $A$ be a primitive axial algebra with $J(A)=0$. Then
\begin{enumerate}
\item[(a)] every block $I_c$, $c\in C$, is a simple primitive axial algebra; and
\item[(b)] $A=\oplus_{c\in C}I_c$.
\end{enumerate}
\end{proposition}

\begin{proof}
We first establish (a). Let $I=I_c$ be an arbitrary block. Since $J(A)=0$ and $I\neq 
0$, there exists a maximal ideal $J$ not containing $I$. Then $A=I+J$, since $J$ is 
maximal. Furthermore, $I/I\cap J\cong I+J/J=A/J$ is a simple algebra, which means 
that $I\cap J$ is a maximal ideal of $I$. If $J'$ is any other maximal ideal of $A$ 
then either $I\subseteq J'$ or $I\not\subseteq J'$ and then, similarly, $I\cap J'$ is 
a maximal ideal of $I$. 

We claim that necessarily $I\cap J=I\cap J'$. Suppose that this is not the case, that 
is, $I\cap J\neq I\cap J'$. Then $I=I\cap J+I\cap J'$. By Lemma \ref{sum}, $a\in 
I\cap J$ or $a\in I\cap J'$. In either case, we get a contradiction, because $I$ is 
not the smallest ideal of $A$ containing $a$. Thus, $I\cap J=I\cap J'$.

We have shown that $I\cap J$ is contained in all other maximal ideals $J'$ of $A$. 
Hence, $I\cap J\subseteq J(A)=0$. This gives us that $A=I\oplus J$, since 
$IJ\subseteq I\cap J=0$. However, this means that every ideal of $I$ is an ideal of 
$A$. If $K\neq 0$ is a proper ideal of $I$ then $K+J>J$ is a proper ideal of $A$; a 
contradiction, since $J$ is a maximal ideal. This shows that $I$ has no proper non-
zero ideals, that is, $I$ is simple, proving (a).

For (b), we first note that, for distinct blocks $I_c$ and $I_{c'}$, we have that 
$I_c\cap I_c'=0$, as it is an ideal in both $I_c$ and $I_{c'}$. Hence also 
$I_cI_{c'}\subseteq I_c\cap I_{c'}=0$.

We already know by Lemma \ref{block decomposition} that $A=\sum_{c\in C}I_c$. To see 
that this is a direct sum, suppose we have a non-trivial linear relation 
$\sum_{i=1}^ku_i=0$, where $u_1,u_2,\ldots,u_k$ are elements of pairwise distinct 
blocks $I_{c_1},I_{c_2},\ldots,I_{c_k}$. Assuming that $k$ is smallest possible, we 
must have that each $u_i$ is non-zero. Furthermore, $u_k=-u_1-u_2-\ldots-u_{k-1}\in 
J=I_{c_1}+I_{c_2}+\ldots+I_{c_{k-1}}$. On the one hand, this means that $I_{c_k}\cap 
J\neq 0$ and, since $I_{c_k}$ is simple, this yields $I_{c_k}\subseteq J$. On the 
other hand, $I_{c_k}J=0$ by the previous paragraph. Taking an axis $a\in c_k$, we now 
see that $a=a^2\in I_{c_k}J=0$; this contradiction proves (b).
\end{proof}

As a consequence, we have the following.

\begin{corollary} \label{modulo Jacobson}
If $A$ is a primitive axial algebra then $A/J(A)$ is a direct sum of simple primitive 
axial algebras.
\end{corollary}

We can now identify the hull-kernel topology of $A$. 

\begin{proposition}
The hull-kernel topology of a primitive axial algebra $A$ is discrete.
\end{proposition}

This is a consequence of Corollary \ref{modulo Jacobson}. Clearly, there is a natural 
bijective correspondence between the maximal ideals of $A$ and of $\bar A=A/J(A)$. 
Hence we can assume that $J(A)=0$ and so $A=\oplus_{c\in C}I_c$ is a direct sum of 
simple algebras. From this, it can be shown that every ideal of $A$ is of the form 
$\oplus_{c\in S}I_c$ for a subset $S\subseteq C$. Consequently, the lattice of all 
ideals of $A$ is the same as the lattice of all subsets of $C$, and this implies 
in turn that the hull-kernel topology is discrete. Note that this argument is quite 
general and not specific to axial algebras. Hence we skip the details.  

\section{Further questions}

In this final section, we discuss our results and mention some related questions 
that remain open. Our main resut is that in every primitive axial algebra $(A,X)$ 
admitting a Frobenius form the three radicals, the form radical $A^{\perp}$, the 
Jacobson radical $J(A)$ and the axial radical $R(A)$ are included in one another as 
follows: 
$$
R(A)\subseteq J(A)\subseteq A^\perp.
$$ 
Furthermore, under a mild and natural extra condition that $(a,a)\neq 0$ for all 
$a\in X$, we have that $R=J(A)=A^\perp$. Can the same ultimate conclusion be reached 
without this extra condition? Note that $A^\perp$ can be artificially inflated. 
Indeed, the Frobenius form induced on the semisimple factor algebra $A/J(A)$ can be 
chosen individually on each simple summand. In particular, we can make the Frobenius 
form zero on all summands but one, in which case $A^\perp$ is definitely bigger than
$J(A)$ as long as $A/J(A)$ is not simple. Hence, we do not believe that the equality 
$J(A)=A^\perp$ can be shown without the additional condition above.

On the other hand, we do not see any simple tricks allowing to make $R(A)$ and $J(A)$ 
distinct. Based on this we pose the following question.

\begin{question}
Is it true that the axial radical $R(A)$ and the Jacobson radical $J(A)$ coincide in 
every primitive axial algebra? Equivalently, is it true that $J(A)$ never contains 
primitive axes?
\end{question}

Note that here we are not assuming that $A$ admits a Frobenius form, as it seems 
unrelated to the main point of this question. In fact, the existence of a non-zero 
Frobenius form seems to be a property of the factor algebra $A/J(A)$. Namely, $A$ 
admits such a Frobenius form if and only if at least one simple summand of $A/J(A)$ 
admits such a form.

Related to the above question are the following questions concerning the structure of 
blocks and the domination relation.

\begin{question}
Is it true that each axial block is indecomposable?
\end{question}

\begin{question}
Can there be non-trivial embeddings between blocks, i.e., can the dominance relation 
to be non-symmetric?
\end{question}

We also want to mention the relation of our results and the above questions to the 
concept of the non-annihilation graph introduced in \cite{kms}.

\begin{definition}
Suppose that $(A,X)$ is a primitive axial algebra. The \emph{non-annihilation graph} 
$\Dl$ has $X$ as its set of vertices and distinct $a,b\in X$ are adjacent in the 
non-annihilation graph if and only if $ab\neq 0$.  
\end{definition}

The following question was first posed in \cite{kms} in the context of Majorana 
algebras and then also repeated in a more general seting in the survey \cite{ms}. 

\begin{question}
Is it true that in the finest (direct) sum decomposition of a primitive axial algebra 
the summands correspond to the connected components of the non-annihilation graph 
$\Dl$?
\end{question}

When $A$ is decomposed as a direct sum of primitive axial algebras then, clearly, 
each connected component of $\Dl$ is fully connected in one of the summand algebras. 
hence the issue is whether an indecomposable primitive axial algebra can have a 
disconnected non-annihilation graph $\Dl$.

We have shown that, in the class of Majorana algebras, indecomposable algebras are 
simple. Hence the above question is equivalent to the following.

\begin{question}
Is it true that the non-annihilation graph $\Dl$ of a simple Majorana algebra is 
necessarily connected?
\end{question}

\Addresses
	  
\end{document}